\documentclass{article}
\usepackage{url}
\usepackage{graphicx}
\graphicspath{{}}
\usepackage[tbtags]{amsmath}%amsthm is included in amsmath package centertags
\usepackage{amsfonts}
\usepackage{mathtools}
\mathtoolsset{showonlyrefs=fasle,showmanualtags}
\usepackage{amsthm} % use amsthm envirenments
\theoremstyle{plain}% default
\newtheorem{theorem}{Theorem}
\newtheorem{lemma}[theorem]{Lemma}

\newtheorem{corollary}[theorem]{Corollary}
\theoremstyle{definition}
\newtheorem{definition}{Definition}

\theoremstyle{remark}
\newtheorem*{remark}{Remark}

\usepackage{amssymb}
\usepackage{algorithm}
\usepackage{algorithmic}
\usepackage{float}

\usepackage[english]{babel}
\usepackage{enumitem}
\usepackage{color,booktabs,colortbl}
\usepackage{latexsym,bm}
\usepackage{multicol}
\usepackage{comment}
\usepackage{geometry}
\usepackage{fancyhdr}
\usepackage{caption}
\newcommand{\norm}[1]{\left\lVert#1\right\rVert}

\newcommand{\R}{\mathbb{R}}

\newcommand{\rd}{\mathrm{d}}
\DeclareMathOperator*{\argmin}{arg\,min}

\DeclareMathOperator*{\re}{Re}

\DeclareMathOperator*{\diag}{diag}

\renewcommand{\vec}{\operatorname{vec}}
\usepackage{listings}
\usepackage[backend=biber,citestyle=numeric-comp,sortcites,bibstyle=mybibstyle,firstinits=true,isbn=false,doi=false]{biblatex}
\DeclareFieldFormat{labelnumberwidth}{#1\addperiod}
\DeclareFieldFormat[article]{volume}{\bibstring{volume}\addnbspace\bfseries #1}
\DeclareFieldFormat[article]{number}{\bibstring{number}\addnbspace #1}
\renewbibmacro*{volume+number+eid}{%
  \printfield{volume}%
  \setunit{\addcomma\space}%<---- was \setunit*{\adddot}%
  \printfield{number}%
  \setunit{\addcomma\space}%
  \printfield{eid}}
\AtEveryCitekey{%
  \ifentrytype{article}
    {\clearfield{title}}
    {}%
  \clearfield{doi}%
  \clearfield{url}%
  \clearlist{publisher}%
  \clearlist{location}%
  \clearfield{note}%
}

\DeclareFieldFormat{urldate}{\bibstring{urlseen}\space#1}
\renewbibmacro*{date}{% Based on date bib macro from chem-acs.bbx
  \iffieldundef{year}
    {\ifentrytype{online}
       {\printtext[urldate]{\printurldate}}
       {\printtext[date]{\printdate}}}
    {\printfield[date]{year}}}

\DeclareFieldFormat*{title}{#1}
\DeclareFieldFormat[book]{date}{\textbf{#1}} % doctorate added this line
\renewbibmacro*{title}{% Based on title bib macro from biblatex.def
  \ifboolexpr{ test {\iffieldundef{title}}
               and test {\iffieldundef{subtitle}} }
    {}
    {\ifboolexpr{ test {\ifhyperref}
                  and not test {\iffieldundef{doi}} }
       {\href{http://dx.doi.org/\thefield{doi}}
          {\printtext[title]{%
             \printfield[titlecase]{title}%
             \setunit{\subtitlepunct}%
             \printfield[titlecase]{subtitle}}}}
       {\ifboolexpr{ test {\ifhyperref}
                     and not test {\iffieldundef{url}} }
         {\href{\thefield{url}}
            {\printtext[title]{%
               \printfield[titlecase]{title}%
               \setunit{\subtitlepunct}%
               \printfield[titlecase]{subtitle}}}}
         {\printtext[title]{%
            \printfield[titlecase]{title}%
            \setunit{\subtitlepunct}%
            \printfield[titlecase]{subtitle}}}}%
     \newunit}%
  \printfield{titleaddon}%
  \clearfield{doi}%
  \clearfield{url}%
  \clearlist{language}% doctorate added this
  \clearfield{note}% doctorate added this
  \ifentrytype{article}% Delimit article and journal titles with a period
    {\adddot}
    {}}

\begin{document} \title{Numerical Optimization Algorithm of Wavefront Phase Retrieval from Multiple Measurements} \author{Ji
Li\thanks{\href{mailto:liji597760593@126.com}{\nolinkurl{liji597760593@126.com}},
LMAM, School of Mathematical Sciences, Peking University, Beijing
100871, China}\quad Tie
Zhou\thanks{\href{mailto:tzhou@math.pku.edu.cn}{\nolinkurl{tzhou@math.pku.edu.cn}},
LMAM, School of Mathematical Sciences, Peking University, Beijing
100871, China}}
\date{\today}
\maketitle
\begin{abstract}
Wavefront phase retrieval from a set of intensity measurements can be
formulated as an optimization problem. Two nonconvex objective models
(MLP and its variants LS) based on maximum likelihood estimation are
investigated. We develop numerical optimization algorithms for
real-valued function of complex variables and apply them to solve the
wavefront phase retrieval problem efficiently. Numerical simulation is
given with application to three wavefront phase retrieval problems. LS
model shows better numerical performances than MLP model. An
explanation for this is that the distribution of the eigenvalues of
Hessian matrix of LS model is more clustered than MLP model. LBFGS
shows more robust performance and takes fewer calculations than other
line search methods.
\end{abstract}
%\tableofcontents
\section{Introduction}
Phase retrieval has been a long-standing inverse problem in imaging
sciences such as X-ray crystallography~\cite{Millane1990},
electron microscopy~\cite{Misell1973}, X-ray diffraction
imaging~\cite{Shechtman2015}, optics~\cite{Kuznetsova1988} and
astronomy~\cite{Fienup1982}, just name a few. Those
imaging techniques are based on the Fourier transform relation between
the optical field and its diffraction pattern. Only the intensity of the
optical field can be recorded by CCD (Charge Coupled Device), while the
more important phase information is lost in the course of
measurement. Mathematically speaking, phase retrieval in continuous setting is to find a function
$u(\bm{x})$ when given its Fourier transform amplitude
$\lvert\mathcal{F}(u(\bm{x}))\rvert$. If we have retrieved the phase
of Fourier transform, then the function can be easily obtained by
inverse Fourier transform, hence the name phase retrieval.

The phase retrieval problem is the
so-called inverse problem. Violation of continuous dependence on data is
the usual difficulty in numerical solutions for most inverse problems
and so is the phase problem. However, for phase retrieval problem,
nonuniqueness is the more important source of ill-conditionedness. Clearly, if
$u(\bm{x})$ is a solution to the phase retrieval problem, then $cu(\bm{x})$ for any scalar $c\in\mathbb{C}$ obeying $\lvert c\rvert
=1$, $u(\bm{x}-\bm{x}_0)$ for any given $\bm{x}_0$ and
$\overline{u(\bm{x}-\bm{x}_0)}$ are also solutions. Those ``trivial
associates'' are acceptable ambiguities from the physical viewpoint. The
results are also applicable in discrete setting. References for
theoretical results in terms of uniqueness of the phase retrieval
problem can be found in~\cite{Akutowicz1956,Akutowicz1957,Barakat1984}
for continuous setting and
in~\cite{Sanz1985,Hayes1982,Bates,Sanz1983}
for discrete model. Those references point out that the solution is
almost \emph{relatively unique} in multidimensional cases, except for
``trivial associates'' solutions.  In literature~\cite{Luke2002}, it is pointed
out that, these uniqueness results are of fundamental importance, but
it does not apply to numerical algorithms, in particular in the
presence of noise. Nonuniqueness in the absolute sense, i.e., all
solutions are same only up to a global phase shift, may be the main
reason for the failure and common stagnation problem of practical algorithm~\cite{Fannjiang2012a}.

Two approaches to mitigate the nonuniqueness of phase retrieval
problem are using a priori information and using multiple
measurements. The former approach relies on all kinds of a
priori information about the signal, such as real-valudedness~\cite{Fienup1982},
positivity and support constraints. Practical reconstruction methods, such as Error Reduction (ER) and its variants (HIO~\cite{Fienup1982}, HPR and
RAAR~\cite{Luke2004,Luke2003}), are developed. While those algorithms are simple
to implement and amenable to additional constraints, they all take the
disadvantages of slow convergence and easy stagnation~\cite{Fienup1986}, usually
hundreds of iterations are needed for a reasonable solution. Besides,
the oversampling method~\cite{Miao2000} is used to get a satisfying
solution in above algorithms.

Instead of assuming constraints on the signal, the concept
of using multiple measurements is introduced to phase retrieval, such
as multiple diffraction patterns~\cite{Candes2013} and
ptychography~\cite{Qian2014}. Taking multiple measurements usually yields uniqueness
solution to phase retrieval. It has shown that $\mathcal{O}(N)$ Gaussian
measurements yield the unique solution with high probability for a
signal
$\bm{u}\in\mathbb{C}^N$~\cite{Candes2012}. PhaseLift~\cite{Candes2012}
which solves a semidefinite programming is proposed by lifting the
vector $\bm{u}$ to a rank one matrix $\bm{u}\bm{u}^*$. So it is
prohibitive to two-dimensional phase retrieval problem.

Optical wavefront phase retrieval problem we study here arises in
optics and astronomy as a special case of phase retrieval problem, where two
moduli including intensity of the wavefront and its Fourier transform
intensity are given. Its algorithms have been
widely used in adaptive optics to correct the aberrations of large
telescopes, allowing them to obtain the diffraction-limited image. Light
wave field distorts when passing through the perturbation atmospheric,
leading the image blurring. Usually, the wavefront is parametrized with Zernike
polynomials coefficients. But the Zernike-based parametrization has
the disadvantage that we do not know how many terms are enough to fit
the wavefront. Moreover, there is no analytic expression of Zernike
polynomials for irregular domain, which is usually used in modern telescopes, such as Hubble Space
Telescope and James Webb Space Telescope (JWST).

In this paper, we consider the wavefront phase retrieval problem from
multiple measurements. In optics, multiple measurements (called
diversity images) are easy to
obtain at different positions (called image planes) through optical
axis. This configuration is called phase diversity, which was proposed by
Gonsalves~\cite{Gonsalves1976} in 1976 is to overcome the
ill-posedness of the wavefront phase retrieval problem. Misell method~\cite{Misell1973} as the natural extension
of ER algorithm was used when there are
more image planes, with the same disadvantages as ER algorithm. We formulate
a direct error metric nonconvex optimization problem from a set of diversity images. The discrete scheme takes the advantage of allowing for the most
accurate representation of the domain~\cite{Luke2002}.

The \emph{main contributions} in this paper consist three points:
\begin{itemize}
\item \emph{Formulation of Objective Functions:} Two objective
  functions (called MLP and LS models) are investigated. MLP model is
  deduced from the negative
logarithm maximum likelihood estimation objective according to the
Poisson photon distribution. The least squares (LS) model used in
iterative projection method then can be interpreted as an
approximation of the MLP model.
\item \emph{Extension of Optimization Method of Function of complex variables:} We develop the line
search method for real-valued function of complex variables, which
directly operates the complex gradient. It is more convenient than
operating the real and image parts respectively. Since 
LBFGS algorithm scales well for large scale problem and shows more
robust performance, we deduce the LBFGS
algorithm, which directly operates on complex gradient. 
\item \emph{Eigenvalue Analysis of Hessian Matrix:} We analyze the
eigenvalues of the complex Hessian matrix for the two objective functions. From
the analysis, the distribution of eigenvalues of Hessian
matrix of LS model is generally more clustered than MLP model, which gives an
intuitive explanation for the better numerical performance.
\end{itemize} 

The plan of this paper is as follows. First in
Section~\ref{sec:form-wavefr-reconstr} we state the formulation of the wavefront phase retrieval problem in
phase diversity setting. In Section~\ref{sec:three-misf-object}, we present two
nonconvex objectives to minimize for the solution of the wavefront
phase retrieval problem. The gradient and
Hessian operators are deduced in
Section~\ref{sec:frechet-derivation}. Then we derive some line search
methods for the unconstrained real-valued function of complex
variables in Section~\ref{sec:unconstr-optim-meth} based on the
$\mathbb{C}$-$\mathbb{R}$ calculus (see
subsection~\ref{sec:compl-grad-hess}) and the global
convergence to a stationary point are proven. In
Section~\ref{sec:least-squares-model}, we give an intuitive
explanation on why LS model converges faster by analyzing the eigenvalues of
Hessian matrices of the two objectives. In Section~\ref{sec:numerical-test}, we apply several line
search algorithms to three wavefront phase retrieval problems, i.e.,
Zernike type, von Karman atmospheric type for disc pupil and JWST type
for segmented pupil. LBFGS algorithm 
shows the most robust performance than other line search methods. Conclusion comes last in
Section~\ref{sec:conclusion}.

To be clear, the notations and conventions used in the paper are
briefly summarized here. We use script letter $\mathcal{F}$ for Fourier transform and $F$ for
discrete Fourier transform. Symbols $(x,y)$ and $(\xi,\eta)$
denotes the coordination in signal plane and Fourier plane
respectively. For every complex vector $\bm{z}\in \mathbb{C}^n$,
$\lvert\bm{z}\rvert\in\R^n$ denotes its element-wise magnitude
vector. $\bm{a}\circ\bm{b}$ and $\bm{a}/\bm{b}$ denote their element-wise
Hadamard product and quotient of matrices or vectors $\bm{a}$ and
$\bm{b}$. $\bar{\bm{z}}$ denotes the element-wise complex conjugate of
$\bm{z}$.

\section{Mathematical Model}
\label{sec:form-wavefr-reconstr}

\subsection{Model of Optical Image Formation}
\label{sec:image-model}
We begin with an isoplanatic Fourier optics model for a simple
telescope with incoherent light and denote the optical wave field
(wavefront) as $u(\bm{x}),\bm{x}=(x,y)$. The relation between intensity
$I$ recorded in the focus plane and object
$\psi$ is given by
\begin{equation}
  \label{eq:2} I(\xi,\eta)=\iint_{\R^2}\lvert
U(\xi-x,\eta-y)\rvert^2\lvert
\psi(x,y)\rvert^2\rd x\rd y,
\end{equation}where we cancel out the phase shift before the integral
with the reason that it does not change the intensity. The kernel in Equation~\eqref{eq:2} is
\begin{equation*} \lvert U(\xi,\eta)\rvert^2 = \lvert
\mathcal{F}(u(x,y))\rvert^2
\end{equation*}with coordination scaling. The kernel characterizes the optical system, and
$\lvert U(\xi,\eta)\rvert^2$ is known as the \emph{point spread
  function} (PSF) of the optical system. Equation~\eqref{eq:2} is the
basis of Fourier optics, the strict deduction from the Helmholtz
equation and Huygens-Fresnel principle can be referred in
book~\cite{Goodman2005} or the review paper~\cite{Luke2002}. If the
intensity is recorded in the defocus plane, there is some phase shift
in the kernel, see \eqref{eq:4}.

If the optical system is in a
homogeneous medium, the wavefront $u(x,y)$ is the indicator
function of lens region $A$, i.e.,
\begin{equation*} u(\bm{x})=\chi_A(\bm{x}) =
  \begin{cases} 1,& \bm{x}\in A,\\ 0,&\text{otherwise}.
  \end{cases}
\end{equation*}

\subsection{Wavefront Phase Retrieval Problem}
\label{sec:wavefr-reconstr-prob}
In most applications, however, the assumption of homogeneity is not
correct for the optical field. The resulting deviations are called phase aberrations. Deviations may occur at any
point along the path of propagation and can be caused by atmosphere
turbulence in an intervening medium or geometric flaws of lens.
Aberration resulting from atmosphere turbulence can be simulated with
the 
von Karman model based on a realization of a strictly stationary
Gaussian stochastic process. In such case, the wavefront $u(\bm{x})$
is not the indicator function but a complex function within the lens
region. It can be represented in polar form
\begin{equation}
  \label{eq:3} u(x,y) = a(x,y)\exp(i\phi(x,y)),
\end{equation} where $a(x,y)$ denotes the transmitting function and
$\phi(x,y)$ is the aberration phase.
 
Equation~\eqref{eq:2} becomes $I = \lvert\mathcal{F}(u)\rvert^2$ when
object $\psi(\bm{x})$ is a $\delta$ function. In astronomy and
optical design, the object is assumed a point optical source. Then the
wavefront phase retrieval problem arises when recovering generalized
pupil function $u(\bm{x})$ from the intensity $I$ and intensity $\lvert u\rvert^2$. Wavefront phase retrieval is the first
step in adaptive optics technique to measure and reshape the wavefront
phase. Then the reconstructed wavefront is used to compensate the 
wavefront aberrations induced by atmospheric turbulence. And it is
also used to optical test to characterize the flaws of lens. To
overcome the ill-posedness of the wavefront phase retrieval
problem, the phase diversity method is
used by providing multiple measurements to further restrict the
wavefront $u(\bm{x})$.

In phase diversity setting, diversity image intensities are 
observed at defocus $d_m,m=1,2,\ldots,L-1$ (in wavelength) are given
as
\begin{align}
  \label{eq:4}
  I_m(\xi,\eta) &= \lvert \mathcal{F}(u;d_m)\rvert^2\nonumber\\
   &=\left\lvert \iint_{\R^2}u(x,y)\exp\left[i2\pi
  d_m(x^2+y^2)\right]\exp\left[-i2\pi(\xi x+\eta y)\right]\rd x
\rd y\right\rvert^2.
\end{align}

In summary, the mathematical model of wavefront phase retrieval with
phase diversity measurements is to find the wavefront function $u$
when given the intensity $\lvert u\rvert^2=I_0$ and phase diversity image
intensities $I_m$, $m =
1,2,\ldots,L-1$.

Numerical algorithms to wavefront phase retrieval problem is based on discretization.
In the following, we focus on the discrete version of the 
problem. Function $u$ is replaced by a vector $\bm{u}$. Given
the data set $\bm{I}_0^{\text{obs}}$ and $\bm{I}_m^{\text{obs}}$, $m=1,2,\ldots,L-1$, we formulate the following optimization problem
\begin{equation}
  \label{eq:11}
  \bm{u} := \argmin_{\bm{u}\in \mathbb{C}^N} \sum_{m=0}^{L-1}\mathcal{S}\left(\bm{I}_m^{\text{obs}},\bm{I}_m(\bm{u})\right)= \argmin_{\bm{u}\in\mathbb{C}^N}\sum_{m=0}^{L-1}\mathcal{E}_m(\bm{u}),
\end{equation}
where $\mathcal{S}(\cdot,\cdot)$ is the data misfit functional, $\bm{I}_m$ are the predicted data from image model
$\bm{I}_m=\lvert F(\bm{u};d_m)\rvert^2$ and $\bm{I}_m^{\text{obs}}$
are the observed data.

\section{Choice of the Data Misfit Functional}
\label{sec:three-misf-object}
\subsection{Maximum Likelihood Estimation}
\label{sec:maxim-likel-estim}

The physics of CCD for recording the intensity of incoming wave field
is a photon counting process, the noise resulted from the process and
other noise, such as external/internal background radiation noise and
thermoelectric noise, all yield Poisson process due to photon
counting. Readout noise which is modeled by Gaussian distribution can
be lowered to zero~\cite{Snyder1993}, so we only consider the Poisson
noise model in CCD\@.

For brevity, we use the one-dimensional notation to describe the
essence.\footnote{It should be two-dimensional for wavefront and
diversity images in discrete forms, and we can stack the array in
column to get a one-dimensional vector.} The observed photon counting
data in diversity images $\bm{I}_m^{\text{obs}},m=0,1,\ldots,L-1$ are
described by random vectors $\bm{\mathfrak{I}}_m\in \R^N,
m=0,1,\ldots,L-1$. The components of every vector are random variables
yielding independent Poisson distribution with mean $\bm{I}_m\in
\R^N$, where $\bm{I}_m$ should be equal to the predicted data
$\bm{I}_m=\lvert F(\bm{u};d_m)\rvert^2$ from the image formulation
model when given the wavefront $\bm{u}$. The observed data
$\bm{I}_m^{\text{obs}}=(I_{m,1}^{\text{obs}},I_{m,2}^{\text{obs}},\ldots,I_{m,N}^{\text{obs}})\in\R^N$
is a sample of random vector $\bm{\mathfrak{I}}_m$.

Without loss of generality, we omit the subscript $m$ for deducing the
misfit functional. Assuming that the number of photons $n_i$ is
proportional to $I_i^{\text{obs}},i=1,2,\ldots,N$. The probability of
recording intensity $I_i^{\text{obs}}$ (or receiving $n_i$ photons) is
\begin{equation*}
  p(\mathfrak{I}_i=I_i^{\text{obs}};I_i) = \frac{(I_i)^{I_i^{\text{obs}}}}{I_i^{\text{obs}}\, !}e^{-I_i}.
\end{equation*}
By independent property, the joint probability density is given by
\begin{equation*}
  p(\bm{\mathfrak{I}}=\bm{I}^{\text{obs}};\bm{I}) = \prod_{i=1}^N \frac{(I_i)^{I_i^{\text{obs}}}}{I_i^{\text{obs}}\, !}e^{-I_i}.
\end{equation*}
The negative log-likelihood function associated with the Poisson
distribution can be cast as
\begin{equation}
  \label{eq:12}
  \mathcal{L}(\bm{I}|\bm{I}^{\text{obs}}) = \sum_{i=1}^N I_i-I_i^{\text{obs}}\log I_i,
\end{equation}
where we omit the $\log\left(I_i^{\text{obs}}\,!\right)$ constant
offset in the summand.\footnote{If we normalize the misfit functional
by subtracting the minimal value
$I_i^{\text{obs}}-I_i^{\text{obs}}\log I_i^{\text{obs}}$, we obtain
the Kullback-Leibler divergence
$\mathbb{K}\mathbb{L}\left(\bm{I}^{\text{obs}},\bm{I}\right)=\sum_{i=1}^{N}I_i-I_i^{\text{obs}}-I_i^{\text{obs}}\log\left(\frac{I_i}{I_i^{\text{obs}}}\right)$.}
So the data misfit functional in discrete model can be given by
\begin{equation}
  \label{eq:13}
  \mathcal{S}(\bm{I}^{\text{obs}},\bm{I}) = \sum_{i=1}^N  [K(\bm{u})]_i-I_i^{\text{obs}}\log\bigl([K(\bm{u})]_i\bigr),
\end{equation}
where $K(\bm{u})=\lvert F(\bm{u};d)\rvert^2$ for a given $d$.

Equation~\eqref{eq:12} can be written as
\begin{equation}
  \label{eq:14}
  \mathcal{L}(\bm{I}|\bm{I}^{\text{obs}}) =\widetilde{\mathcal{L}}(\bm{M}|\bm{M}^{\text{obs}}) = \sum_{i=1}^N M_i^2-(M_i^{\text{obs}})^2\log(M_i^2),
\end{equation}
by substituting $M_i= \sqrt{I_i}$ into~\eqref{eq:12},
where $M_i,M_i^{\text{obs}}$ is pointwise amplitude. 

Expand~\eqref{eq:14} around $\bm{M}_i =
\bm{M}_i^{\text{obs}}$, leading to 
\begin{equation}
  \label{eq:15}
  \widetilde{\mathcal{L}}(\bm{M}|\bm{M}^{\text{obs}})\simeq 2\sum_{i=1}^N(M_i-M_i^{\text{obs}})^2+\text{const.}.
\end{equation}
Equation~\eqref{eq:15} can be regarded as an approach to approximating
the Poisson likelihood function $\mathcal{L}$, and
$\widetilde{\mathcal{L}}$ corresponds to a Gaussian distribution about
$M_i$ with a variance of $1/4$. So another data misfit function can be
proposed
\begin{equation}
  \label{eq:16}
  \mathcal{S}(\bm{I}^{\text{obs}},\bm{I}) = \widetilde{\mathcal{S}}(\bm{M}^{\text{obs}},\bm{M})=\norm{\widetilde{K}(\bm{u})-\bm{M}^{\text{obs}}}_2^2,
\end{equation}
where $\widetilde{K}(\bm{u})=\lvert F(\bm{u};d)\rvert$.  It is worth
to note that data misfit function~\eqref{eq:16} is fortunately
identical to the cost function found throughout the phase retrieval
literature. Fienup~\cite{Fienup1982} considered ER method as a
fixed-step gradient descent method applied to this cost function,
which can be clearly seen in Section~\ref{sec:frechet-derivation}.
\begin{remark}
If we expand~\eqref{eq:12} around $\bm{I}=\bm{I}^{\text{obs}}$,
leading to $ \mathcal{L}(\bm{I}|\bm{I}^{\text{obs}}) \simeq
\frac{1}{2}\sum_{i=1}^N
\frac{\left(I_i-I_i^{\text{obs}}\right)^2}{I_i^{\text{obs}}}$. This
data misfit functional is recommended in~\cite{Isernia1999}. However,
in our numerical tests this functional is not so good to use, so we do
not consider it in the paper.
\end{remark}

\subsection{Two Objective Models}
\label{sec:two-models}

We have obtained the data misfit functional for every diversity image,
then we can just add up every misfit function with equal or different
weights, here the equal weight scheme is used for brevity. By looking at~\eqref{eq:13} and~\eqref{eq:16}, if there exists
an index such that $\lvert
F(\bm{u};d)\rvert_i$ closes to zero, then its gradient is singular at
this point, so we need some perturbation applied to objectives to
overcome possible overflow in numerical algorithms. 

Two data misfit models are presented to measure the error that
quantifies inconsistency between the given data and the predicted data
from the image formulation model~\eqref{eq:4}. Hereafter the
superscript $(\cdot)^{\text{obs}}$ are omitted without
ambiguity. Since equation~\eqref{eq:16} can be rewritten as
\begin{equation}
  \label{eq:1}
   \mathcal{S}(\bm{I}^{\text{obs}},\bm{I}) = \sum_{i=1}^N\biggl([K(\bm{u})]_i+M_{i}^{2}-2\lvert
     K(\bm{u})\rvert_iM_{i}\biggr).
\end{equation}

We consider the following two
models for the error metric $\mathcal{E}_m(\bm{u})$ in the paper:
\begin{itemize}
\item Maximum Likelihood Poisson (MLP) model:
  \begin{equation}
    \label{eq:18}
    \mathcal{J}_m(\bm{u})=\sum_{i=1}^N\biggl( [K_m(\bm{u})]_i-I_{m,i}\log\bigl([K_m(\bm{u})]_i+\epsilon^2\bigr)\biggr),
  \end{equation}
\item Least Squares (LS) model:
  \begin{align}
    \widetilde{\mathcal{J}}_m(\bm{u})&=\sum_{i=1}^N\biggl([K_m(\bm{u})]_i-2\sqrt{
      [K_m(\bm{u})]_i+\epsilon^2}M_{m,i}\biggr),   \label{eq:19}
  \end{align}
\end{itemize}
where operators $K_m\colon \mathbb{C}^N\to \R^N$, $K_m(\bm{u})=\lvert
F(\bm{u};d_m)\rvert^2$, $m=1,2,\ldots,L-1$ are the forward formulation
of predicted data, and $K_0(\bm{u})=\lvert\bm{u}\rvert^2$ is the
measurement of pointwise amplitude of the wavefront.

Having obtained the misfit objective functional, in the framework of
optimization we should derive the gradient and Hessian to utilize the
line search method. In the next subsection, we use the Fr\'{e}chet
derivative to reduce the corresponding gradient and Hessian for MLP and
LS models respectively.

\subsection{Deduction of Gradient and Hessian Operator by Fr\'{e}chet Derivative}
\label{sec:frechet-derivation}

In this section, the gradient and Hessian operator (Hessian-vector
multiplication) of the two objectives are derived. We deduce those
expressions according to Fr\'{e}chet derivative. Because of the similarity
of every image plane misfit functional, we take single image data to
derive the expressions for brevity with notation omitted subscript
$m$. All the above objectives~\eqref{eq:18} and~\eqref{eq:19}
can be regarded as functional defined on Hilbert
space $\mathbb{C}^N$ to $\R$. The
gradient and Hessian operator of two objectives models to minimize
are presented in the following. For details of deduction of MLP
and LS models see Appendix~\ref{sec:grad-hess-oper}
and~\ref{sec:grdi-hess-oper}.

We use symbol $F_m$ and $F_m^*$ to denote the
Fourier transform $F(\bm{u};d_m)$ with defocus $d_m$ and adjoint
operator respectively. For MLP model~\eqref{eq:18}, its gradient and
Hessian-vector multiplication are given by
\begin{subequations}
\begin{align}
  \nabla \mathcal{J}_m(\bm{u}) &= F_m^*\left(F_m(\bm{u})\circ
    \left(\bm{1}-\frac{\bm{I}_m}{\lvert F_m(\bm{u})\rvert^2+\epsilon^2}\right)\right), \label{eq:30}\\
 \mathcal{H}_{\mathcal{J}_m}[\bm{u}](\bm{h}) &= F_m^*\left(\left(\bm{1}-\frac{\epsilon^2\bm{I}_m}{\left(\lvert F_m(\bm{u})\rvert^2+\epsilon^2\right)^2}\right)\circ
F_m(\bm{h})\right)\\
&\phantom{=}\qquad+F_m^*\left(\frac{\bm{I}_m\circ
    F_m(\bm{u})^2}{\left(\lvert F_m(\bm{u})\rvert^2+\epsilon^2\right)^2}\circ \overline{F_m(\bm{h})}\right).  \label{eq:31}
\end{align}
\end{subequations}

For LS model~\eqref{eq:19}, its gradient and Hessian-vector multiplication are given by
\begin{subequations}
\begin{align}
   \nabla \widetilde{\mathcal{J}}_m(\bm{u}) &=\bm{u}-F_m^*\left(\frac{\bm{M}_m}{\sqrt{\lvert
      F_m(\bm{u})\rvert^2+\epsilon^2}}\circ F_m(\bm{u})\right), \label{eq:32}\\
 \mathcal{H}_{\widetilde{\mathcal{J}}_m}[\bm{u}](\bm{h}) &= F_m^*\left(\left(\bm{1}-\frac{\bm{M}_m}{2\sqrt{\lvert
       F_m(\bm{u})\rvert^2+\epsilon^2}}\circ\left(\frac{\lvert
       F_m(\bm{u})\rvert^2+2\epsilon^2}{\lvert
       F_m(\bm{u})\rvert^2+\epsilon^2}\right)\right)\circ
 F_m(\bm{h})\right)\notag\\
&\phantom{=}\qquad+F_m^*\left(\frac{F_m(\bm{u})^2\circ \bm{M}_m}{2\left(\lvert
    F_m(\bm{u})\rvert^2+\epsilon^2\right)^{3/2}}\circ \overline{F_m(\bm{h})}\right).\label{eq:33}
\end{align}
\end{subequations}
We see that the gradient $\nabla \widetilde{\mathcal{J}}_m(\bm{u})$ is
identical to the projection onto the Fourier transform constraints in ER
algorithm for the $m$th diversity image with a perturbed $\epsilon^2$.

It is noted that gradient and Hessian operators are similar to the
complex gradient and Hessian-vector multiplication in the framework of
$\mathbb{C}$-$\mathbb{R}$ theory (see
Section~\ref{sec:compl-grad-hess}). It is more easy to deduce the
above gradient and Hessian expression by following the Fr\'{e}chet
derivative instead of the definition~\eqref{eq:23} and~\eqref{eq:24}.The theory deals with real-valued
functions of complex variables. The objectives here are the cases. In next section, we develop several optimization methods which are
similar to standard optimization method of real variables. It is directly
implemented by complex gradient, not by representing the complex
gradient in its
real and imaginary parts.

\section{Unconstrained Optimization Method}
\label{sec:unconstr-optim-meth}

In this section, we consider to minimize the real-valued function $f(\bm{z})$,
where $\bm{z}\in\mathbb{C}^N$. The optimization method is
usually carried out with respect to the real and imaginary parts of
these complex variables. This approach is used in literature~\cite{Luke2002}
for wavefront phase retrieval problem. The reason for this approach is
to avoid difficulties with the definition and interpretation of the
gradient and Hessian with respect to complex
variables~\cite{VanDenBos1994}. In next subsection we review the
$\mathbb{C}$-$\mathbb{R}$ theory of complex gradient and Hessian of
real-valued function of complex variables. Those concepts are first
described by researchers in signal processing, in which a number of
problems with real-valued function of complex variables arise.

\subsection{Complex Gradient and Hessian}
\label{sec:compl-grad-hess}

A complex vector $\bm{z}\in\mathbb{C}^N$ can be represented by
$\bm{z}^{\mathcal{C}}=(\bm{z}^T,\bm{\bar{z}}^T)^T\in\mathbb{C}^{2N}$ or
$\bm{z}^{\mathcal{R}}=(\bm{x}^T,\bm{y}^T)^T\in\mathbb{R}^{2N}$ with its real parts and
image parts. Formally the conjugate coordinate derivatives are defined as
\begin{subequations}
\begin{align}
  \frac{\partial f}{\partial \bm{z}} :&= \frac{\partial
    f(\bm{z},\overline{\bm{z}})}{\partial
    \bm{z}}\big|_{\overline{\bm{z}}=\text{const.}} = \left(\frac{\partial
      f}{\partial z_1},\frac{\partial
      f}{\partial z_2},\ldots,\frac{\partial
      f}{\partial z_N}\right), \label{eq:21}\\
  \frac{\partial f}{\partial \overline{\bm{z}}} :&= \frac{\partial
    f(\bm{z},\overline{\bm{z}})}{\partial
    \overline{\bm{z}}}\big|_{\bm{z}=\text{const.}} = \left(\frac{\partial
      f}{\partial \overline{z_1}},\frac{\partial
      f}{\partial \overline{z_2}},\ldots,\frac{\partial
      f}{\partial \overline{z_N}}\right).\label{eq:22}
\end{align}
\end{subequations}
Those definitions follow standard notation from multivariate calculus
in which derivatives are row vectors and gradients are column vectors. 
\begin{definition}
Assume that all partial derivatives of a real-valued function $f$ of
complex variables $\bm{z}$ with respect to $\bm{x}$ and $\bm{y}$
exist. The complex gradient of $f$ is defined as 
\begin{equation}
  \label{eq:23}
  \nabla f=\left(\frac{\partial f}{\partial \bm{z}}\right)^*
  =\left(\frac{\partial f}{\partial \overline{\bm{z}}}\right)^T.
\end{equation}
Similarly, we define
\begin{equation*}
  H_{zz} := \frac{\partial }{\partial
    \bm{z}}\left(\frac{\partial f}{\partial \bm{z}}\right)^{*},\quad  H_{\bar{z}z} := \frac{\partial }{\partial
    \bar{\bm{z}}}\left(\frac{\partial f}{\partial \bm{z}}\right)^{*},\quad  H_{z\bar{z}} := \frac{\partial }{\partial
    \bm{z}}\left(\frac{\partial f}{\partial \bar{\bm{z}}}\right)^{*},\quad
 H_{\bar{z}\bar{z}} := \frac{\partial }{\partial
    \bar{\bm{z}}}\left(\frac{\partial f}{\partial \bar{\bm{z}}}\right)^{*},
\end{equation*}
then the complex Hessian matrix is given by
\begin{equation}
\label{eq:24}
  H^{\mathcal{C}}_f(\bm{z}):=
  \begin{pmatrix}
    H_{zz}& H_{\bar{z}z}\\
H_{z\bar{z}}&  H_{\bar{z}\bar{z}}
  \end{pmatrix}.
\end{equation}
\end{definition}
Actually, $f$ is not analytic, for the reason that it is not satisfied the Cauchy-Riemann
 condition. Without ambiguity, the subscript $f$ and superscript
 $\mathcal{C}$ of $H$ may be omitted. If $f$ is real-valued, then
\[\overline{\frac{\partial f}{\partial z_i}}=\frac{\partial
    f}{\partial \overline{z_i}}, \text{ and }  H_{zz} =
H_{zz}^*, H_{\bar{z}z} =H_{z\bar{z}}^*=\overline{H_{z\bar{z}}},
H_{\bar{z}\bar{z}} = \overline{ H_{zz}}.\]
 Because of those equalities, we only need store $\frac{\partial
  f}{\partial \bm{z}}$, $H_{zz}$ and
$H_{\bar{z}z}$~\cite{Kreutz-Delgado2009}.

Optimization method is deduced from the Taylor's approximation
expansion, which takes the form
\begin{align}
   f(\bm{z}+\Delta \bm{z}) &= f(\bm{z}) + 2\re\left\langle \Delta
   \bm{z},\nabla f(\bm{z})\right\rangle +
   \re\left\langle \Delta
   \bm{z},H_{zz}\Delta \bm{z} + H_{\bar{z}z}\overline{\Delta
   \bm{z}}\right\rangle+ \text{h.o.t}.\label{eq:28}
\end{align}
From~\eqref{eq:28}, we can define the Hessian operator or
Hessian-vector multiplication at point $\bm{z}$ of function $f$:
\begin{equation}
  \label{eq:29}
  \mathcal{H}_f[\bm{z}](\bm{h}) = H_{zz}\bm{h} + H_{\bar{z}z}
   \overline{\bm{h}}.
\end{equation}
The above definition matches the Fr\'{e}chet derivative as indicted in \eqref{eq:54}.

\subsection{Line Search Methods}
\label{sec:line-search-methods}

Considering the minimization of a real-valued function $f(\bm{z})$ of
complex variables, it is more convenient to operate the complex
gradient directly. The straightforward extension for optimization of
function of complex variables details in this section. Line search
methods construct a sequence
\begin{equation}
  \label{eq:37}
  \bm{z}_{k+1} = \bm{z}_k+\alpha_k\bm{d_k}.
\end{equation}
The basic method is first to choose a descent direction
$\bm{d}_k\in\mathbb{C}^N$, then to minimize a one-dimensional optimization
problem with some line search scheme to find the step length
$\alpha_k\in\mathbb{R}$ at $k$th
iteration. 

A direction $\bm{d}_k$ is called a descent direction, if it satisfies 
\begin{equation}
  \label{eq:38}
  \re(\bm{d}_k^*\bm{g}_k) <0,
\end{equation}
where $\bm{g}_k = \nabla f(\bm{z}_k)$. With holding global convergence to
local minimizer, step length $\alpha_k$ is not arbitrary. There is
usually a requirement of $\alpha_k$, which are known as Wolfe conditions:
\begin{subequations}
  \begin{equation}
    \label{eq:39}
    f(\bm{z}_k + \alpha_k \bm{d}_k)\leq f(\bm{z}_k) + c_1\alpha_k \re(\bm{d}_k^*\bm{g}_k)
  \end{equation}
and
\begin{equation}
  \label{eq:40}
  \re (\bm{d}_k^*\bm{g}_{k+1}) \geq c_2 \re(\bm{d}_k^*\bm{g}_k),
\end{equation}
\end{subequations}
where condition $0<c_1<c_2<1$ is satisfied. Generally we take $c_1 =
10^{-4}$, $c_2 = 0.9$ as commended in
book~\cite{Nocedal2006}. Equations~\eqref{eq:39} and~\eqref{eq:40} are
known as the sufficient decrease and curvature condition
respectively. The former ensures sufficient decrease of the objective function, and the latter ensures the
gradient converges to zero.

According to the choices of descent direction $\bm{d}_k$, we introduce
some common line search methods. The simple case is the steepest
descent method, which takes $\bm{d}_k = -\bm{g}_k$. The steepest
descent (SD) method choose the fastest descent direction at every
iteration, but its asymptotic rate of convergence is inferior to other
methods. For poorly conditioned problems, it increasingly `zigzags' as
the gradients point nearly orthogonally to the shortest direction to a
minimum point. To accelerate the rate of convergence, nonlinear
conjugate gradient (NCG) method is widely used. The conjugate gradient
direction $\bm{d}_k$ is generated by the recurrence relation
\begin{equation}
  \label{eq:41}
  \bm{d}_k = -\bm{g}_k + \beta_k \bm{d}_{k-1},
\end{equation}
where $\bm{d}_0 = \bm{0}$. There are a variety of options to choose
parameter $\beta_k$ for nonlinear problem~\cite{Steihaug1983}. In this
paper, we take the Hestenes-Stiefel form
\begin{equation}
  \label{eq:42}
  \beta_k^{HS} = -\frac{\re\left(\bm{g}_k^*(\bm{g}_k-\bm{g}_{k-1})\right)}{\re\left(\bm{d}_{k-1}^*(\bm{g}_k-\bm{g}_{k-1})\right)}.
\end{equation}

Since we have the exact Hessian-vector multiplication~\eqref{eq:31} and~\eqref{eq:33}, so we can
implement the truncated Newton (TN) method, where the Newton direction
is obtained by solving the equation
\begin{equation*}
  \mathcal{H}_f[\bm{z}_k]\bm{d}_k=-\bm{g}_k
\end{equation*}
 using conjugate gradient method. Unfortunately, the corresponding Hessian
matrix is not always positive definite. So in this case, we should
adopt the negative curvature direction, see~\cite[Chapter
6]{Nocedal2006}. Newton method is considered as a two-order method
provided initial point is in the neighborhood of the 
minimizer. It is shown in our numerical test that TN method is not robust and performs worse compared with other methods.

We deduce the LBFGS method for optimization of function
of complex variables. 
LBFGS has advantages of better rate of convergence and less memory need
than BFGS method. It is as well as easily implemented by two-loop recursion
with vector-vector inner product. We begin with the BFGS formula of updating
the approximation Hessian matrix.

 If the number of variables $N$ is large, the cost of storing and
manipulating the matrix $B_k^{\mathcal{C}}$ become prohibitive. The famous
LBFGS method
is appropriate for large-scale problem, and  the descent direction $\bm{d}_k$ can be obtained by the easy
two-loop recursion, which is described in
Algorithm~\ref{alg:1}.
\begin{algorithm}[H]
  \caption{L-BFGS two-loop recursion\label{alg:1}}
  \begin{algorithmic}
    \REQUIRE $\bm{g}_k$, $\bm{s}_i=\bm{z}_{i+1}-\bm{z}_i$,
    $\bm{y}_i=\bm{g}_{i+1}-\bm{g}_i$, $\rho_i =
    \frac{1}{\re(\bm{y}_i^*\bm{s}_i)}$, for $i=k-m,\ldots,k-1$,
    \ENSURE $\bm{d}$, such that $\bm{d}^{\mathcal{C}}=-B_k^{\mathcal{C}}\bm{g}_k^{\mathcal{C}}$
\STATE $\bm{d}\leftarrow -\bm{g}_k$
\FOR {$i=k-1,k-2,\ldots,k-m$}
\STATE $\alpha_i = \rho_i\re(\bm{s}_i^*\bm{d})$
\STATE $\bm{d}\leftarrow \bm{d}-\alpha_i \bm{y}_i$
\ENDFOR
\STATE $\bm{d}\leftarrow \gamma\bm{d}$, with the scaling suggested
by Shanno and Phua $\gamma=\frac{\re(\bm{y}_{k-1}^*
  \bm{s}_{k-1})}{\bm{y}_{k-1}^*\bm{y}_{k-1}}$
\FOR {$i=k-m,k-m+1,\ldots,k-1$}
\STATE $\beta\leftarrow \rho_i\re(\bm{y}_i^*\bm{d})$
\STATE $\bm{d}\leftarrow \bm{d}+(\alpha_i-\beta)\bm{s}_i$
\ENDFOR
  \end{algorithmic}
\end{algorithm}

Its performance depends on the number of storing pair
$\{\bm{s}_i,\bm{y}_i\}$ $m$, and usually $m$ takes 10, 20. In wavefront
phase retrieval problem, the numerical performance is not sensitive to $m$, in
which $m$ takes 2. It is shown that its performance scales
well and takes the priority over the NCG method in our numerical tests.

Another way to achieve global convergence is to use the trust region
technique to determine a search direction and step length
simultaneously. The approach can be easily extended to function of
complex variables. We omit the extension for the reason that its
performance is similar to truncated Newton method.

It is noted that the major differences between the optimization of
function of real variables and optimization of function of complex
variables are that we use
the complex gradient instead of its real and imaginary parts and we take the real parts of some complex
quantities. The generalization of the line search method to complex
gradient can be proven in the framework of $\mathbb{C}$-$\mathbb{R}$
calculus. It is easy to verify the correction, we omit the details with
consideration of paper limitation.

\subsection{Global Convergence of Wavefront Phase Retrieval}
\label{sec:glob-conv-wavefr}

The below theorem states the global convergence to a stationary point
of function of complex variables. Its proof is the same as the version
for function of real variables.
\begin{theorem}
  Suppose that $f\colon \mathbb{C}^N\to\mathbb{R}$ is bounded below in
  $\mathbb{C}^N$ and that $f$ is continuously differentiable. And
  assume that the gradient $\nabla f$ is Lipschitz continuous, that
  is, there exists a constant $L>0$ such that
  \begin{equation}
    \norm{\nabla f(\bm{u})-\nabla f(\bm{v})}_2\leq
    L\norm{\bm{u}-\bm{v}}_2,\qquad \forall\quad \bm{u},\bm{v}\in\mathbb{C}^N,
\label{eq:45}
  \end{equation}
The iteration $\bm{z}_{k+1}=\bm{z}_k+\alpha_k\bm{d}_k$, where
$\bm{d}_k$ is a descent direction, i.e., $\re\left(\bm{d}_k^*\nabla
  f\right)<0$, and $\alpha_k$ satisfies the Wolfe conditions, then
$\bm{z}_k$ converges to a stationary point $\bm{z}^{\dagger}$ of $f$,
i.e., $\nabla f(\bm{z}^{\dagger})=\bm{0}$. 
\end{theorem}
\begin{proof}
We denote the gradient $\nabla f(\bm{z}_{k})$ as $\bm{g}_k$.
From~\eqref{eq:40} we have that
\begin{equation*}
  \re \left(\bm{d}_k^*(\bm{g}_{k+1}-\bm{g}_k)\right) \geq (c_2-1) \re(\bm{d}_k^*\bm{g}_k),
\end{equation*}
and the Lipschitz condition~\eqref{eq:45} implies that
\begin{equation*}
  \re \left(\bm{d}_k^*(\bm{g}_{k+1}-\bm{g}_k)\right)\leq \alpha_kL\norm{\bm{d}_k}_2^2.
\end{equation*}
Combining these two relations, we obtain
\begin{equation*}
  \alpha_k\geq \frac{c_2-1}{L}\frac{\re(\bm{d}_k^*\bm{g}_k)}{\norm{\bm{d}_k}_2^2}.
\end{equation*}
From the Wolfe descent condition,
  \begin{equation*}
    f_{k+1}\leq f_{k} - c_1\frac{1-c_2}{L}\frac{(\re(\bm{d}_k^*\bm{g}_k))^2}{\norm{\bm{d}_k}_2^2}.
  \end{equation*}
By summing this expression over $j=0,1,\ldots,k$, we obtain
\begin{equation*}
  f_{k+1}\leq f_0-c\sum_{j=0}^k\frac{(\re(\bm{d}_k^*\bm{g}_k))^2}{\norm{\bm{d}_k}_2^2},
\end{equation*}
where $c=c_1(1-c_2)/L$. Since $f$ is bounded below, we have that
\begin{equation*}
  \bm{g}_k\to 0,\qquad k\to \infty,
\end{equation*}
which concludes the proof.
\end{proof}
Next theorem states that the gradient of the MLP model and LS model are Lipschitz.
\begin{theorem}
\label{thm:2}
 The gradient of function $\mathcal{J}_m,\widetilde{\mathcal{J}}_m$
 are global Lipschitz for all $m=0,1,\ldots,L-1$, and the
constants are $L_m=1+\norm{\bm{I}_m}_{\infty}/\epsilon^2$ and
$\widetilde{L}_m=1+\norm{\bm{M}_m}_{\infty}/\epsilon$ for MLP and LS
model respectively. Considering the physical setting of wavefront phase retrieval problem
in phase diversity setting,
given the data set $\bm{I}_i,i=0,\ldots,L-1$ or
$\bm{M}_i,i=0,\ldots,L-1$, the overall optimization
functional~\eqref{eq:11} with both MLP and LS models are Lipschitz.
\end{theorem}

\begin{corollary}
  Since MLP and LS model have the property that its gradient is
  Lipschitz, so the steepest descent method is global convergence to
  a stationary point.
\end{corollary}

\section{Least Squares Model is Better}
\label{sec:least-squares-model}
In this section, we do analyze the eigenvalues of the complex Hessian
matrix of the two models. It is motivated by the fact that the rough
rate of convergence of optimization methods depends on the
distribution of eigenvalues of the Hessian.

We can form the Hessian matrices of MLP model and LS model. One-dimensional discrete Fourier transform $F(\bm{u})$ can be
expressed in matrix-vector multiplication as $F(\bm{u}) = W\bm{u}$,
where $W$ is the corresponding DFT matrix. Without loss of
generality, we assume the problem is one-dimensional.\footnote{Define
operator $\vec(\cdot)$ is to stack a two-dimensional array to a column
vector. For two-dimensional discrete Fourier transform, its compact
matrix form is given by $\vec(F(\bm{u}))=(W^T\otimes W)\vec(\bm{u})$,
where $\otimes$ is the Kronecker product.} The diversity image
formulation operator $F_m(\bm{u})$ can be written in matrix form
$F_m(\bm{u})=WD_m\bm{u}=U_m\bm{u}$, where diagonal matrix $D$ ($\lvert
d_{ii}\rvert=1$) signs the phase shift from the phase
diversity.\footnote{The unitary matrix $U$ is not the optical field.}
We first analyze the eigenvalues of every diversity image term
$\mathcal{E}_m$.

The complex Hessian matrices are directly obtained from their Hessian
operators respectively:
\begin{itemize}
\item For MLP model, its Hessian matrix is given by
\begin{equation}
  \label{eq:46}
  H_m^{MLP} =
  \begin{pmatrix}
    U_m^*\diag\left(\bm{1}-\frac{\epsilon^2\bm{I}_m}{\left(\lvert F_m(\bm{u})\rvert^2+\epsilon^2\right)^2}\right)U_m& U_m^*\diag\left(\frac{\bm{I}_m\circ
    (F_m(\bm{u}))^2}{\left(\lvert F_m(\bm{u})\rvert^2+\epsilon^2\right)^2}\right)U_m^*\\
  U_m\diag(*)U_m&U_m\diag(*)U_m^*
  \end{pmatrix};
\end{equation}
\item For LS model, its Hessian matrix
\begingroup
\everymath{\scriptstyle}
\scriptsize
\begin{equation}
  \label{eq:47}
  H_m^{LS} =
  \begin{pmatrix}
   U_m^{*}\diag\left(\bm{1}-\frac{\bm{M}_m}{2\sqrt{\lvert
       F_m(\bm{u})\rvert^2+\epsilon^2}}\circ\left(\frac{\lvert
       F_m(\bm{u})\rvert^2+2\epsilon^2}{\lvert
       F_m(\bm{u})\rvert^2+\epsilon^2}\right)\right)U_m&U_m^{*}\diag\left(\frac{(F_m(\bm{u}))^2\circ \bm{M}_m}{2\left(\lvert
    F_m(\bm{u})\rvert^2+\epsilon^2\right)^{3/2}}\right)U_m^{*}\\
 U_m\diag(*)U_m&U_m\diag(*)U_m^*
  \end{pmatrix};
\end{equation}
\endgroup
\end{itemize}
For convenience, we use symbol $\diag(*)$ to denote the diagonal
matrix without concern for the specific content. For the particular form
of the Hessian matrix $H_m$, the following theorem is straightforward.
\begin{theorem}
  Given vectors $\bm{r}\in\mathbb{R}^N$, and $\bm{c}\in\mathbb{C}^N$,
  then the eigenvalues of Hermitian matrix 
  \begin{equation*}
    H_1 =
    \begin{pmatrix}
      \diag(\bm{r})&\diag(\bm{c})\\
      \diag(\overline{\bm{c}})&\diag(\bm{r})
    \end{pmatrix}\in \mathbb{C}^{2N}
  \end{equation*}
are real and can be worked out as $r_i\pm \lvert
c_i\rvert,i=1,2,\ldots,N$. Then Hermitian matrix
\begin{equation*}
  H_2 = \begin{pmatrix}
      U^*\diag(\bm{r})U&U^*\diag(\bm{c})U^*\\
      U\diag(\overline{\bm{c}})U&U\diag(\bm{r})U^*
    \end{pmatrix}\in \mathbb{C}^{2N}
\end{equation*}
with unitary matrix $U\in\mathbb{C}^{N\times N}$ share the same
eigenvalues of $H_1$.
\end{theorem}
\begin{proof}
Without loss of generality, we assume that $\lvert c_i\rvert\neq 0$.
Because similarity matrices share the same eigenvalues, we utilize the
elementary row transform $P_i=I_{2N}+\frac{c_i}{\lvert
c_i\rvert}\bm{e}_{i,i+N}$ and elementary column transform
$P_i^{-1}=I_{2N}-\frac{c_i}{\lvert c_i\rvert}\bm{e}_{i,i+N}$
$i=1,2,\ldots,N$ to reduce the matrix $H_1$. Iteratively
left-multiplying and right-multiplying $H_1$ by $P_i$ and $P_i^{-1}$
yields
\begin{equation*}
  P_NP_{N-1}\cdots P_2P_1H_1P_1^{-1}P_2^{-1}\cdots P_{N-1}^{-1}P_{N}^{-1}=
  \begin{pmatrix}
    \diag\left(\bm{r}+\lvert\bm{c}\rvert\right)&0\\
    *&\diag\left(\bm{r}-\lvert\bm{c}\rvert\right)
  \end{pmatrix},
\end{equation*}
so eigenvalues of $H_1$ is $\bm{r}\pm\lvert \bm{c}\rvert$. Since 
\begin{equation*}
  H_2 =
  \begin{pmatrix}
    U^*&\bm{0}\\
    \bm{0}&U
  \end{pmatrix}H_1\begin{pmatrix}
    U&\bm{0}\\
    \bm{0}&U^*
  \end{pmatrix},
\end{equation*}
it is obvious that $H_2$ and $H_1$ share the same eigenvalues.
\end{proof}
The following corollary characterizes the eigenvalues of the Hessian
matrix for the two models respectively.
\begin{corollary} 
For the above two models, the eigenvalues of their Hessian matrices
are
\begin{subequations}
  \begin{align}
  \lambda(H_m^{MLP}) &= \left\{1+\frac{(\lvert
        F_m(\bm{u})\rvert_i^2-\epsilon^2)I_{m,i}}{(\lvert
        F_m(\bm{u})\rvert_i^2+\epsilon^2)^2},1-\frac{I_{m,i}}{\lvert
        F_m(\bm{u})\rvert_i^2+\epsilon^2},i=1,2,\ldots,N\right\};\label{eq:49}\\
\lambda(H_m^{LS})&=\left\{1-\frac{\epsilon^2M_{m,i}}{(\lvert
    F_m(\bm{u})\rvert_i^2+\epsilon^2)^{3/2}},1-\frac{M_{m,i}}{\sqrt{\lvert
      F_m(\bm{u})\rvert_i^2+\epsilon^2}},i=1,2,\ldots,N\right\};\label{eq:300}.
  \end{align}
\end{subequations}
And if in every iteration the point $\bm{u}$ satisfies the inequality $F_m(\bm{u})_i\geq
I_{m,i}$ for all $i=1,2,\ldots,N$ and all $m=0,1,\ldots,L-1$, then the
corresponding objective functions are 
convex. 
\end{corollary}
First we fix the $m$ and consider just the matrix $H_m$ under the
assumption that the iteration point $\bm{u}_k$ is in the neighborhood
of the solution to wavefront phase retrieval problem. For
the two models, the distribution of eigenvalues of of LS model is more
clustered than
MLP model. For comparison, we multiply the eigenvalues by two, thus
both the maximum eigenvalues approximate to two, when $\bm{u}$ is the
minimizer such that $F_m(\bm{u})=\bm{I}_m$. After multiplying by two,
the maximum eigenvalue of Hessian of LS model is smaller than two,
while the maximum eigenvalue of Hessian of MLP model is generally
greater than two, if there exists an index $i$ such that $\lvert
F_m(u)\rvert^2_i\leq I_{m,i}$ (It is usually true for the nonconvex
objective). In another direction, the minimum eigenvalue of Hessian
of LS model is smaller than MLP model by the inequality:
\begin{equation*}
  1-\frac{I_i}{\lvert
        F(\bm{u})\rvert_i^2+\epsilon^2}\leq 2-\frac{2M_i}{\sqrt{\lvert
      F(\bm{u})\rvert_i^2+\epsilon^2}} \quad i=1,2,\ldots,N.
\end{equation*}
So LS model should have better performance in numerical simulation,
since the distribution of its eigenvalues is generally more clustered
than MLP model.
This point can be demonstrated in Section~\ref{sec:numerical-test} by
numerical test.

From the eigenvalue analysis, we can prove the gradient is
Lipschitz, and we state the following lemma.
\begin{lemma}
  For a real-valued function $f$ of complex variables, if its complex Hessian
  matrix $H^{\mathcal{C}}$ satisfies $\norm{H^{\mathcal{C}}}_2\leq L$,
  then the gradient $\nabla f$ is Lipschitz with 
  constant $L$.
\end{lemma}

Now it is easy to show the proof of Theorem~\ref{thm:2}.
\begin{proof}[Proof of Theorem~\ref{thm:2}]
From~\eqref{eq:49} and~\eqref{eq:300}, for every diversity image $m$, the norm of
Hessian is controlled by 
\begin{equation*}
  \norm{H_m^{MLP}}\leq 1+\frac{\norm{\bm{I}_m}_{\infty}}{\epsilon^2},\quad
  \norm{H_m^{LS}}\leq 1+\frac{\norm{\bm{M}_m}_{\infty}}{\epsilon}.
\end{equation*}
 By the fact that given Herimitan
  matrices $A$, $B$ and $C=A+B$, the inequalities of eigenvalues hold, i.e.,
\begin{equation*}
  \max{\lambda(C)}\leq \max{\lambda(A)}+\max{\lambda(B)},\qquad
  \min{\lambda(C)}\geq \min{\lambda(A)}+\min{\lambda(B)},
\end{equation*}
So we have the equalities of the norm of overall Hessian: 
\begin{equation*}
  \norm{H^{MLP}}\leq L +
  \sum_{m=0}^{L-1}\frac{\norm{\bm{I}_m}_{\infty}}{\epsilon^2},\quad
\norm{H^{LS}}\leq L +
  \sum_{m=0}^{L-1}\frac{\norm{\bm{M}_m}_{\infty}}{\epsilon}.
\end{equation*}
So the gradients of MLP and LS models are Lipschitz.
\end{proof}
\section{Numerical Test}
\label{sec:numerical-test}

In this section, The facts that the LS model is the best objective and
the performance of LBFGS algorithm is superior to other line search
methods are first demonstrated by numerical simulations for noiseless
data for two different types of wavefront phase retrieval
problems. Then we apply the LBFGS algorithm for minimizing the LS
model objective to three wavefront phase retrieval problems (except the
aforementioned two types, the more physical James Webb Space Telescope
(JWST) model is included). We present the results of three types of
problems for noiseless data and noisy data.

\subsection{Setup and Error Measurement}
\label{sec:setup-error-meas}

The three types of wavefront phase retrieval problems we consider here
are shown in Figure~\ref{fig:1}. Three types have different pupils and
wavefront. For easy reference, we called the types are (i) Zernike
type, (ii) von Karman type and (iii) JWST type respectively. The
pupils are of the regular shape annulus (i) and disc (ii) and
segmented (iii), which the JWST type is the reduced configuration for
the true physical setting. The simulation phase aberration for Zernike
type is presented by Zernike polynomials~\cite{mahajan2007orthonormal}
in annulus pupil. Just for validating the performance, we construct
the phase using 13th annulus basis polynomial with coefficients
$0.1$. Phase aberration for von Karman type is used to simulate
atmosphere perturbation, which can be described by Fourier transform
of a random process. And wavefront for JWST type is from the NIRCam
data.\footnote{The segmented pupil data and phase aberration data are
extracted from software \textsf{jwpsf}, refer the website
\url{http://www.stsci.edu/jwst/software/jwpsf} for details.} PVs
(peak-to-valley) of the wavefront aberration are $0.58$, $1.02$ and
$1.47$ respectively, and RMSs of aberration are $0.10$, $0.19$ and
$0.21$ in unit of wavelength respectively. In those wavefront
phase retrieval problems we assume pointwise amplitude of the wavefront
$\bm{u}$ equal to unity. To recover the wavefront (phase, more
precise), just two diversity images are used, both two are
out-of-focus with defocus $-3$ and $3$ respectively. The perturbation
$\epsilon$ can be taken as small as possible, here we take
$\epsilon=10^{-14}$.
\begin{figure}[!htbp]
  \centering
  \includegraphics[width=\textwidth,scale = .85]{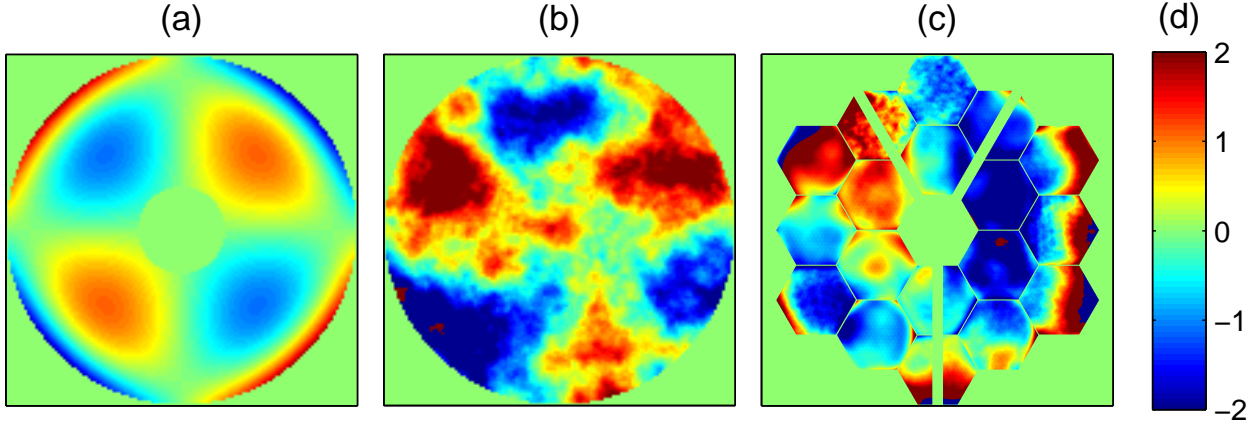}
\caption{Three test types: (a) Zernike (size $128\times 128$), (b) von
  Karman (size $128\times 128$)
   and (c) JWST (size $1024\times 1024$), (d) is the colorbar used through this paper,
if not specified. The pointwise angle (defined in interval
$(-\pi,\pi]$) of the complex wavefront is plotted.}
\label{fig:1}
\end{figure}

We denote the solution to the wavefront phase retrieval problem as
$\bm{u}$, and $\hat{\bm{u}}$ is the returned solution by the line
search method, the relative root mean squared error (RMS) is defined
as 
\begin{equation}
  \label{eq:50}
  RMS = \min_{\lvert c\rvert=1}\frac{\norm{c\bm{u}-\hat{\bm{u}}}_2}{\norm{\bm{u}}_2},
\end{equation}
where $c$ is to get rid of the effect of the constant phase shift of
phase problem. From 
\begin{equation*}
\langle c\bm{u}-\hat{\bm{u}},c\bm{u}-\hat{\bm{u}}\rangle=\lvert
  c\rvert^2\norm{\bm{u}}_2^2-2\re \langle
  c\bm{u},\hat{\bm{u}}\rangle+\norm{\hat{\bm{u}}}_2^2, 
\end{equation*}
the constant $c$ is given by
\begin{equation*}
  c=\frac{\langle \bm{u},\hat{\bm{u}}\rangle}{\lvert \langle \bm{u},\hat{\bm{u}}\rangle\rvert }.
\end{equation*}

In line search methods, for stop criteria, we take the three options:
the number of iterations is $150$, the tolerances for relative change
of objective function $TolFun$ and relative change of variable $TolX$
during one iteration both are $10^{-12}$. The number of stored vector
pairs for LBFGS method $m=2$. Owing to the nonconvexity of the
objectives, the behavior of the algorithms varies considerably
depending on the initialization, hence we take 10 runs for every
algorithm from random initial guesses. Those guesses all have unitary
amplitude in the pupil with random phase uniformly distributed on
$(-\pi,\pi]$.

\subsection{Comparison of Line Search Methods}
\label{sec:comp-line-search}

We first demonstrate our results with the small size problems, i.e.,
Zernike and von Karman types, for noiseless data as a
proof-of-concept. The performance of the four line search methods (SD,
NCG, TN and LBFGS) applied to Zernike and von Karman types based on
the LS model is shown in Figure~\ref{fig:2} for an exemplary run. RMSs
of all algorithms are below $10^{-5}$. The steepest descent method is
comparatively slow in all algorithms, while truncated Newton method have 
the fastest convergence rate near the solution. Note that the number of iterations
of TN method is greater than NCG and LBFGS methods for von Karman
type, the behavior is due to the indefiniteness of the Hessian
matrix. At those points we should adopt the negative curvature
direction, the frequency approximates by $87\%$ in an exemplary run.

In terms of the iterations, NCG method takes fewer iterations in four
algorithms. It is worth noting that the limiting calculation for the
wavefront phase retrieval problem is the Fourier transform, which is
accomplished with the fast Fourier transform (FFT) algorithm. The
average number of FFT calls for the four methods in 10 dependent runs
is shown in Table~\ref{tab:1}. Though NCG method takes the advantage of
fewer iterations, the number of FFT calls is greater than LBFGS
method.\footnote{Certainly, we can reduce the number of iterations by
Preconditioned nonlinear conjugate gradient. But the improvement is
limited.} The large number of FFT calls of TN method results from solving
the Newton direction by CG method, which is the most exhausted
computation.  We also test the easy-implementation Misell
method~\cite{Misell1973} to solve the wavefront phase retrieval
problem, it stagnates earlier and the relative RMS is still above
$0.99$ after $500$ iterations.
\begin{figure}[!htbp]
  \centering
  \includegraphics[width=\textwidth,scale = .85]{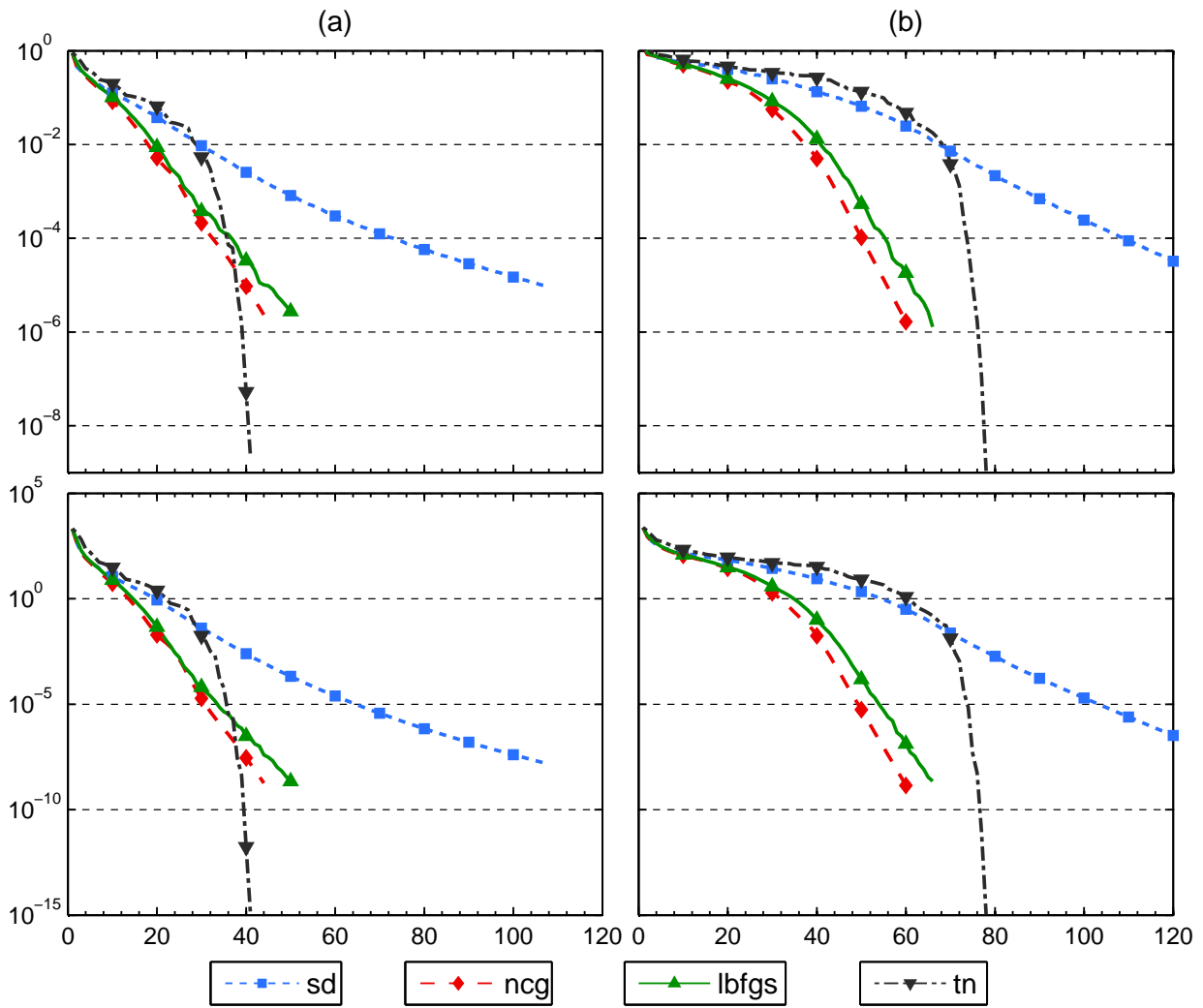}
  \caption{Comparison of algorithms applied to (a) Zernike and (b)
    von Karman data, an exemplary run is shown. Top row plots RMS versus
    iterations, bottom row shows the error misfit of LS model versus
  iterations.}
  \label{fig:2}
\end{figure}

\begin{table}[H]
\caption{Total average number of FFT calls for different methods in 10 dependent runs}
\begin{center}
\begin{tabular}{ccccc}
\toprule
       Case  &      SD  &     NCG  &   LBFGS  &  TN  \\
\midrule
    Zernike  &    1309  &     545  &     299  &    1559  \\
 von Karman  &    1868  &     713  &     418  &    2767  \\
\bottomrule
\end{tabular}
\label{tab:1}
\end{center}
\end{table}

\subsection{Choice of Misfit Functionals}
\label{sec:choice-misf-funct}

How to choose the objective to minimize? It is believed that choosing
appropriate misfit functionals in designing efficient numerical
algorithms for inverse problem. In this subsection, we give the 
performance comparison of the two models (MLP and LS) and the common
used LSI model. The data misfit functional is
just the least squares of the intensities which are directly
recorded. So its objective $\mathcal{E}_m(\bm{u})$ (see equation~\eqref{eq:11}) of LSI model is 
\begin{equation}
    \label{eq:20}
    \widehat{\mathcal{J}}_m(\bm{u})=
   \frac{1}{2}\norm{K_m(\bm{u})-\bm{I}_m}_2^2.
  \end{equation}
Its gradient and Hessian operator are given in the following:
\begin{subequations}
\begin{align}
  \nabla \widehat{\mathcal{J}}_m(\bm{u})&=F_m^{*}\bigl(F_m(\bm{u})\circ\left(\lvert F_m(\bm{u}) \rvert^2-\bm{I}_m\right)\bigr), \label{eq:34}\\
\mathcal{H}_{\widehat{\mathcal{J}}_m}[\bm{u}](\bm{h})&= F_m^{*}\bigl((2\lvert
F_m(\bm{u})\rvert^2-\bm{I}_m)\circ F_m(\bm{h})\bigr)+F_m^{*}\bigl(
F_m(\bm{u})^2\circ\overline{F_m(\bm{h})}\bigr).\label{eq:35}
\end{align}
\end{subequations}.

MLP, LS and LSI models are applied to Zernike and von Karman types is
shown in Figure~\ref{fig:3}. The decrease of LS model is the
fastest and followed by MLP model, LSI model is the slowest. The
theory in Section~\ref{sec:least-squares-model} is
validated by the numerical tests. LSI model is not appropriate for data errors with
a non-Gaussian distribution from the view of statistics. A squared norm data misfit
functional can be interpreted by maximum likelihood estimation for
data with Gaussian noise, see~\cite{Pascarella2011}. MLP model is more
appropriate for Poisson noise. Note that the LSI model is used to
solve the ptychographic phase retrieval problem~\cite{Qian2014}. It
performs well in the thousands of multiple measurements. If there are
not many measurements, the LSI model can not use to solve the phase
retrieval problem.
\begin{figure}[!htbp]
  \centering
  \includegraphics[width=\textwidth,scale = .85]{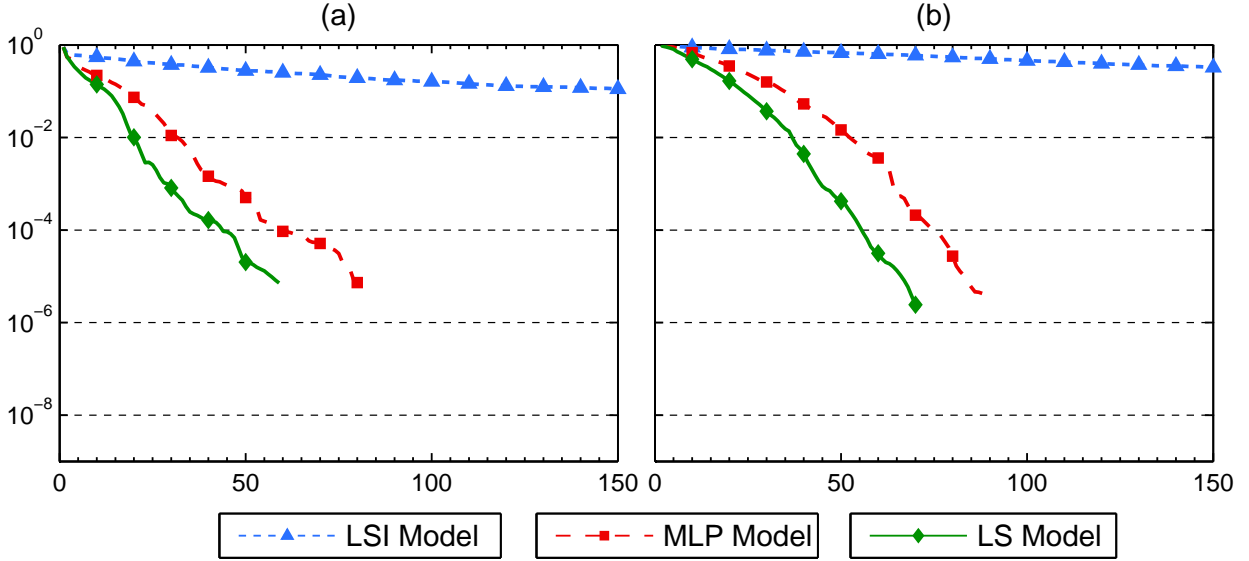}
  \caption{Comparison of different objectives for two simulation data:
    (a) Zernike, (b) von Karman. RMSs of the solution versus iterations
    are plotted.}
  \label{fig:3}
\end{figure}

\subsection{Noisy Measurements}
\label{sec:noisy-measurements}

For noisy case, we use Poisson process to simulate this noise
model. The instability of the wavefront phase retrieval problem with
different defocus is shown in Figure~\ref{fig:4}. When given two
diversity images and the signal-to-noise (SNR) is $10$, the behavior
of instability of inverse problem can be observed. While the error
misfit function of LS model decreases versus iterations. RMS decreases
in first tens of iterations, after arriving a minimum, RMS then
increases in the sequential iterations. So we follow the Mozorov's
discrepancy principles to terminate the iteration, i.e., we should
stop the iteration when the Morozov's postpripori discrepancy
principle has satisfied, like the Landweber iteration for linear
inverse problem. The minima of RMS depend on defocus setting. The
dependent is not sensitive.
\begin{figure}[!htbp]
  \centering
  \includegraphics[width=\textwidth,scale = .85]{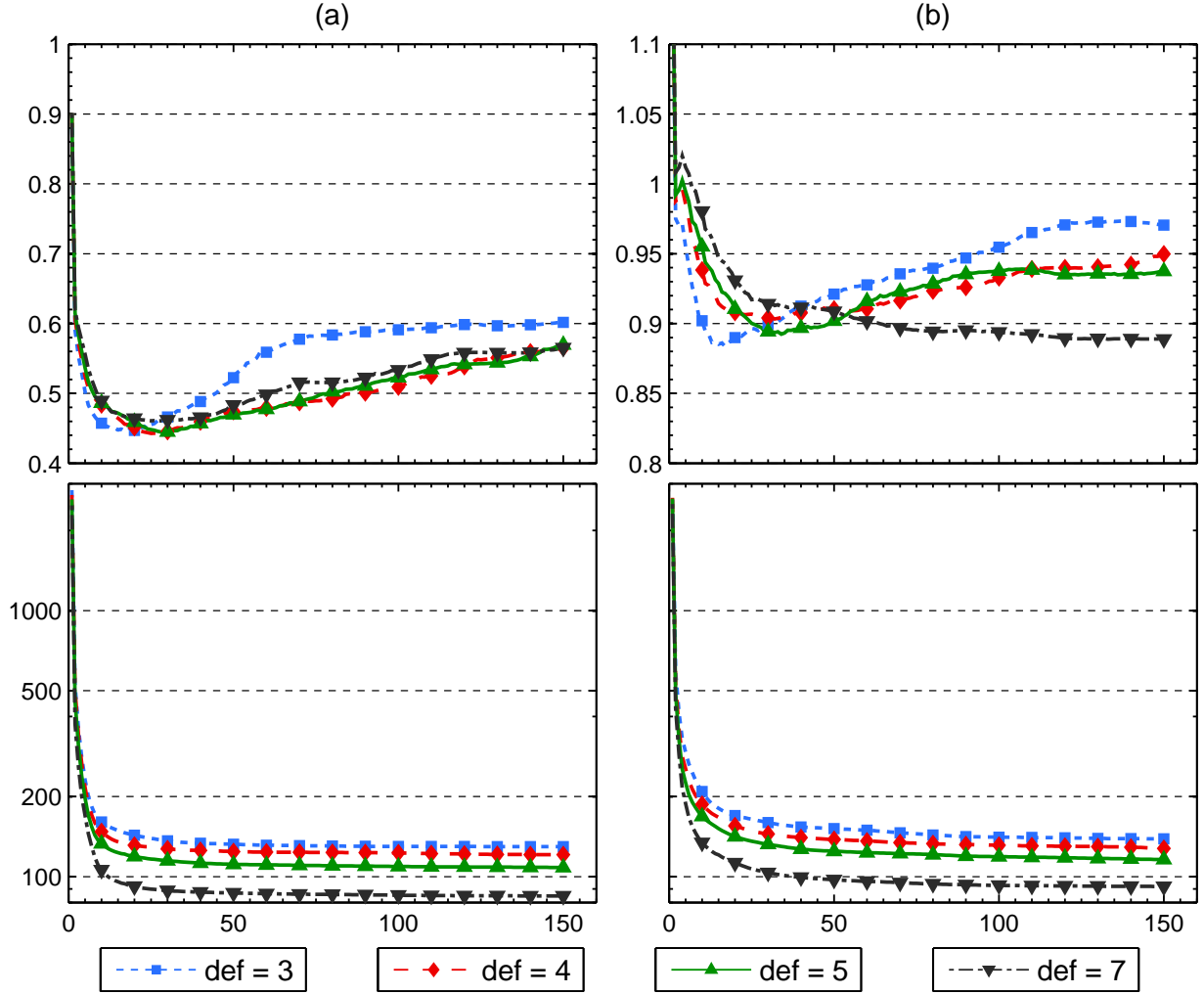}
  \caption{Instability of wavefront phase retrieval problem given two
    diversity images: (a) Zernike, (b) von Karman. RMSs versus
    iterations are
    plotted in the top row, error misfit of LS model versus iteration is
    plotted in the bottom row.}
\label{fig:4}
\end{figure}

With the facts that the best choice for error misfit functional is LS
model and the best algorithm for minimizing LS model is LBFGS
algorithm.  We explore the performance of the LBFGS algorithm applied
to three types wavefront phase retrieval problems for noisy data. We
repeat the LBFGS ten times with different random initials. The two
diversity image data is added random Poisson noise for three different
SNR levels, i.e., $30$, $20$ and $10$. Since the data is noisy, we
terminate the iteration by Morozov's principle. The recovery images
for three SNR levels case and noiseless case are illustrated in
Figure~\ref{fig:6}. To be clear, the wavefront errors for SNR levels
($20$ and $10$) are also illustrated in Figure~\ref{fig:7}. The
minimal values of RMS for the solution are approximately $0.22797$,
$0.10515$ and $0.046935$ for noise levels $SNR=10$, $20$ and $30$,
respectively.  From those RMSs, we can roughly compute the order of
convergence for the wavefront phase retrieval problem, which is used to
describe the rate of the solution in noise case that approximates the
true solution in noiseless case when noise level $SNR$ decays to zero. The
numerical rate of convergence is $0.34$, $0.33$ and $0.20$ for Zernike
type, von Karman type and JWST type, respectively.

\begin{figure}[!htbp]
  \centering
  \includegraphics[width=.85\textwidth,scale = .5]{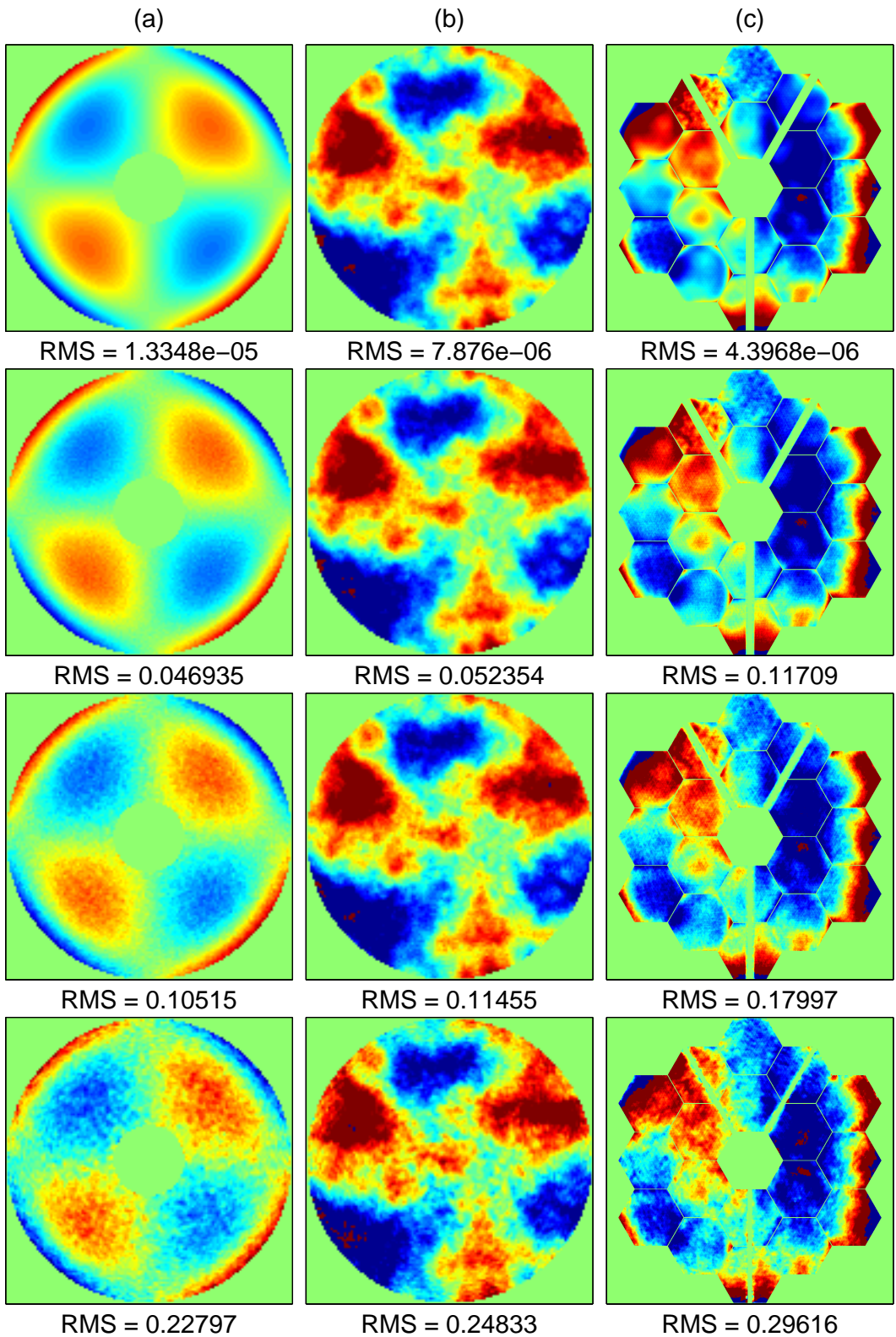}
  \caption{Recovery wavefront for three data: (a) Zernike, (b) von
    Karman, (c) JWST in different SNR level. Top row is without noise,
    then the SNR level decreasing from 30 to 10. RMSs are described 
    below for every figure.}
\label{fig:6}
\end{figure}
\begin{figure}[!htbp]
  \centering
  \includegraphics[width=\textwidth,scale = .85]{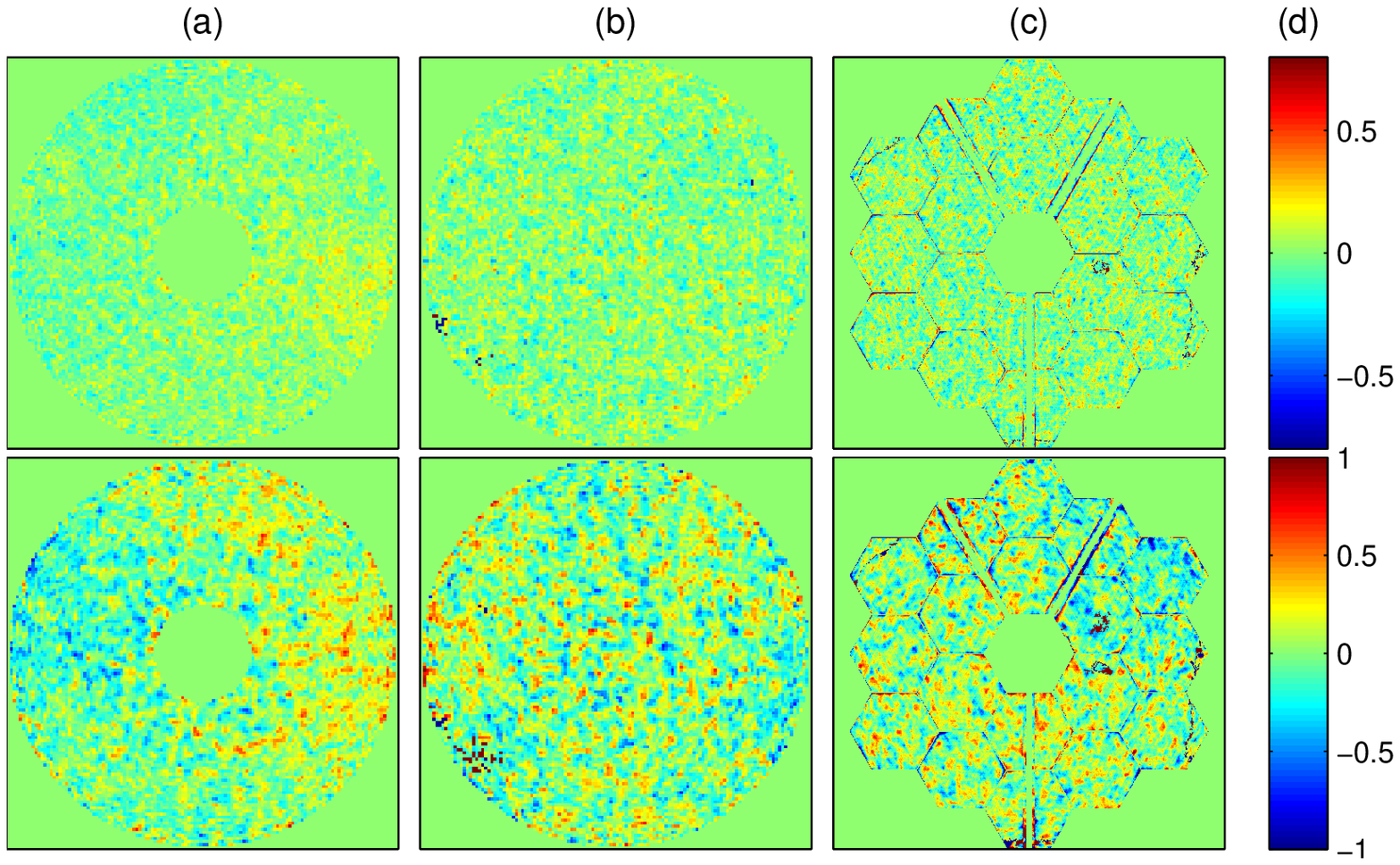}
  \caption{Difference between recovery and original wavefront: (a) Zernike, (b) von
    Karman, (c) JWST for SNR levels 20 (top row), 10 (bottom row), respectively. (d) is the corresponding colorbar.}
\label{fig:7}
\end{figure}

We also explore the numerical performance of LBFGS algorithm
applied to three wavefront phase retrieval problems, when given 
only one diversity image instead. The diversity image data is measured at
defocus $3$. This physical setting has been used by Fienup to
characterize the spherical aberration of Hubble Space Telescope~\cite{Fienup1993a}, in
which Zernike parametrization of the wavefront is used instead of pixel
parametrization.  The solution can also be obtained with more
iterations than in two diversity images case. The algorithm may fail
with some kinds of initials. Numerical solution to wavefront
phase retrieval problem is more easily obtained in the defocus
setting. The uniqueness of the solution
of the problem in defocus setting accounts for the commonly known behavior, see the recent
paper~\cite{Maretzke2015}.
\section{Conclusion}
\label{sec:conclusion}
The wavefront phase retrieval problem given diversity images, a special
example of phase retrieval, is considered in this paper. We formulated the problem as a nonlinear
nonconvex optimization problem with appropriate data misfit
functional.

We proposed three data misfit functionals (MLP and LS models),
that metric the error between noisy data and predicted data from the
formulation model to minimize. MLP model is based on maximum
likelihood estimation on account of the Poisson statistics of data and
LS model is then derived by Taylor expansion. According to Frech\'{e}t derivative,
we deduced the gradient and Hessian operator for two objective models
respectively. These analytic expressions can be incorporated
in optimization algorithm.

We discussed how standard iterative line search methods can be applied
to solve the optimization problem. Since the variables of functional
are complex, it is more convenient to optimize with the complex
gradient than with respect to real and imaginary parts, which
is used in the review paper~\cite{Luke2002}.

After introducing the complex gradient, we derived several line search
methods: the steepest descent and nonlinear conjugate gradient and
truncated Newton and LBFGS algorithm, which were applied to the
optimization problem. We analyzed the eigenvalues of complex Hessian
of three objective models and inferred that the LS model is the best in
all three models in view of its smallest condition number. We
validated our results by numerical simulations and suggested that
LBFGS algorithm give the best performance on a size of $1024\times
1024$ JWST type problem. We found that the LSI model can not be used to
solve the wavefront phase
retrieval problem, while it efficiently solved the ptychographic phase retrieval problem~\cite{Qian2014}. 

The method can be easily extended to the problem that simultaneously
recover the amplitude and phase in phase diversity
setting~\cite{Paxman1992,Brady2006}. As demonstrated in the literature~\cite{Li2013}, PhaseLift often works
much better if a priori knowledge of sparsity of the signals is
incorporated. Despite the already very good performance of the LBFGS
methods for our test problems, it is worth investigating whether
their performance can be improved by exploring the sparsity
structure. The success of the optimization method depends on the exact
known defocus. We found that if the defocus is inaccuracy with the
error above $15\%$, LBFGS algorithm failed. Other approaches for phase
retrieval with multiple measurements, such as introducing random
masks~\cite{Fannjiang2012a,Fannjiang2012}, take the same
disadvantage. So adjoint optimization with respect the wavefront and
defocus is worth investigating in the future.

The LBFGS method can be directly applied to the generalized phase
retrieval problem. It should show better performance than Wirtinger
Flow~\cite{Candes2015}, in which the line search strategy is not
applied. We are carrying on the work, which will be published
elsewhere.
\section*{Acknowledgments}
The first author thanks Heng Mao for introducing him the phase
retrieval problem and Chao Wang for helpful comments and suggestions
on this manuscript draft. The authors wish to thank Zaiwen Wen for
beneficial discussion. This work was partially supported by the
grant.

 \appendix
 \section*{Appendix}
 \label{sec:appendix}
 \setcounter{section}{1}

\subsection{Taylor Expansion for Functional}
\label{sec:tayl-expans-funct}
According to Fr\'{e}chet derivative, the Taylor expansion for functional $\mathcal{J}$ is
\begin{align}
\mathcal{J}(\bm{u}+\bm{h})&=\mathcal{J}(\bm{u})+D\mathcal{J}[\bm{u}](\bm{h})+\frac{1}{2}D^2\mathcal{J}[\bm{u}](\bm{h},\bm{h})+\text{h.o.t},\notag\\
& = \mathcal{J}(\bm{u}) + 2\re\left\langle \bm{h},\nabla \mathcal{J}(\bm{u})\right\rangle + \re\left\langle \bm{h}, \mathcal{H}_{\mathcal{J}}[\bm{u}](\bm{h})\right\rangle
+ \text{h.o.t},\label{eq:54}
\end{align}
where $\nabla \mathcal{J}$ is the gradient operator and
$\mathcal{H}_{\mathcal{J}}[u]$ is the Hessian operator of functional
$\mathcal{J}$ by the Riesz representation theorem.

 \subsection{Gradient and Hessian Operator for MLP Model}
 \label{sec:grad-hess-oper}
We just need to consider the every diversity image individually, so
hereafter we omit the subscript $m$. For MLP model,
 data misfit functional for single image is
\begin{equation}
  \label{eq:55}
  \mathcal{J}(\bm{u}) = \sum_{i=1}^N\left( K(\bm{u})_i-I_i\log(K(\bm{u})_i+\epsilon^2)\right).
\end{equation}
Then introduce $\bm{1}=(1,1,\ldots,1)\in \mathbb{C}^N$, and inner
product $\langle \bm{a},\bm{b}\rangle = \bm{a}^*\bm{b}$, equation~\eqref{eq:55} can be written as
\begin{equation*}
  \mathcal{J}(\bm{u}) = \left\langle \bm{1},K(\bm{u})-\bm{I}\circ\log(K(\bm{u})+\epsilon^2)\right\rangle,
\end{equation*}
where $\log(\cdot)$ is a pointwise function.

According to Fr\'{e}chet derivative and chain rule, we have
\begin{equation}
  \label{eq:56}
  D\mathcal{J}[\bm{u}](\bm{h}) = \left\langle \bm{1},
    DK[\bm{u}](\bm{h})-\frac{\bm{I}}{K(\bm{u})+\epsilon^2}\circ DK[\bm{u}](\bm{h})\right\rangle.
\end{equation}

We substitute
\begin{equation*}
  DK[\bm{u}](\bm{h}) = 2\re\left(F(\bm{u})\circ\overline{F(\bm{h})}\right)
\end{equation*}
into~\eqref{eq:56} resulting
\begin{align*}
  D\mathcal{J}[\bm{u}](\bm{h}) &= \left\langle \bm{1},
    2\re\left(F(\bm{u})\circ\overline{F(\bm{h})}\right)-\frac{\bm{I}}{K(\bm{u})+\epsilon^2}\circ
    2\re\left(F(\bm{u})\circ\overline{F(\bm{h})}\right)\right\rangle\\
&=2\re\left\langle\overline{F(\bm{u})}\circ F(\bm{h}),\bm{1}-\frac{\bm{I}}{K(\bm{u})+\epsilon^2}\right\rangle\\
&=2\re\left\langle \bm{h},F^*\left(F(\bm{u})\circ \left(\bm{1}-\frac{\bm{I}}{K(\bm{u})+\epsilon^2}\right)\right)\right\rangle.
\end{align*}

Hessian operator can be found in a similar way,
\begin{align*}
  D^2\mathcal{J}[\bm{u}](\bm{p},\bm{q}) &= 2\re\left\langle \overline{F(\bm{q})}\circ
  F(\bm{p}),\bm{1}-\frac{\bm{I}}{K(\bm{u})+\epsilon^2}\right\rangle\nonumber
  \\
&\phantom{=}\phantom{2\re}+2\re\left\langle \overline{F(\bm{u})}\circ
  F(\bm{p}),\frac{\bm{I}}{\left(K(\bm{u})+\epsilon^2\right)^2}\circ 2\re\left(F(\bm{u})\circ\overline{F(\bm{q})}\right) \right\rangle\\
&=2\re\left\langle F(\bm{p}),
  \left(\bm{1}-\frac{\bm{I}}{K(\bm{u})+\epsilon^2}\right)\circ
  F(\bm{q})+\frac{\bm{I}\circ
    K(\bm{u})}{\left(K(\bm{u})+\epsilon^2\right)^2}\circ F(\bm{q})\right\rangle\\
&{}+2\re\left\langle F(\bm{p}),\frac{\bm{I}\circ
    F(\bm{u})^2}{\left(K(\bm{u})+\epsilon^2\right)^2}\circ \overline{F(\bm{q})}\right\rangle\\
&= 2\re\left\langle
\bm{p},F^*\left(\left(\bm{1}-\frac{\epsilon^2\bm{I}}{\left(K(\bm{u})+\epsilon^2\right)^2}\right)\circ
F(\bm{q})\right)\right\rangle\\
&{}+2\re\left\langle
\bm{p},F^*\left(\frac{\bm{I}\circ
    F(\bm{u})^2}{\left(K(\bm{u})+\epsilon^2\right)^2}\circ \overline{F(\bm{q})}\right)\right\rangle.
\end{align*}
It can be validated that $D^2\mathcal{J}[\bm{u}](\cdot,\cdot)$ is a bilinear functional
defined on $\mathbb{C}^N\times\mathbb{C}^N$.

From Taylor expression~\eqref{eq:54}, the gradient and Hessian
operator are
given by~\eqref{eq:30} and~\eqref{eq:31}.

\subsection{Gradient and Hessian Operator for LS Model}
\label{sec:grdi-hess-oper}
Misfit function~\eqref{eq:19} can be written as
\begin{equation}
  \label{eq:57}
  \widetilde{\mathcal{J}}(\bm{u}) = \langle F(\bm{u}),F(\bm{u})\rangle -2\langle \sqrt{K(\bm{u})+\epsilon^2},\bm{M}\rangle.
\end{equation}

So derivative~\eqref{eq:57} and follow by chain rule, it yields
\begin{align*}
  D\widetilde{\mathcal{J}}[\bm{u}](\bm{h})&=2\re\left\langle F(\bm{h}),F(\bm{u})\right\rangle -\left\langle
  \frac{2\re(\overline{F(\bm{u})}\circ F(\bm{h}))}{\sqrt{\lvert
      F(\bm{u})\rvert^2+\epsilon^2}}, \bm{M}\right\rangle\\
&=2\re\left\langle \bm{h},\bm{u}-F^*\left(\frac{\bm{M}}{\sqrt{\lvert
      F(\bm{u})\rvert^2+\epsilon^2}}\circ F(\bm{u})\right)\right\rangle.
\end{align*}
From the equation
\begin{equation*}
   D\widetilde{\mathcal{J}}[\bm{u}](\bm{h})=2\re \langle \bm{h},\bm{u}\rangle-2\re\left\langle
   \frac{F(\bm{h})}{\sqrt{\lvert
       F(\bm{u})\rvert^2+\epsilon^2}},F(\bm{u})\circ \bm{M}\right\rangle,
\end{equation*}
\begin{align*}
  D^2\widetilde{\mathcal{J}}[\bm{u}](\bm{p},\bm{q}) &=2\re\langle \bm{p},\bm{q} \rangle
  +2\re\left\langle \frac{\re\left(F(\bm{u})\circ \overline{F(\bm{q})}\right)\circ
    F(\bm{p})}{\left(\lvert
    F(\bm{u})\rvert^2+\epsilon^2\right)^{3/2}},F(\bm{u})\circ
  \bm{M}\right\rangle\\
&{}-2\re\left\langle\frac{F(\bm{p})}{\sqrt{\lvert
       F(\bm{u})\rvert^2+\epsilon^2}},F(\bm{q})\circ
   \bm{M}\right\rangle\\
&=2\re\left\langle\bm{p},\bm{q}-F^*\left(\frac{\re\left(F(\bm{u})\circ \overline{F(\bm{q})}\right)}{\left(\lvert
    F(\bm{u})\rvert^2+\epsilon^2\right)^{3/2}}\circ F(\bm{u})\circ
F(\bm{M})\right)\right\rangle\\
&{}-2\re\left\langle\bm{p},F^*\left(\frac{F(\bm{q})\circ \bm{M}}{\sqrt{\lvert
       F(\bm{u})\rvert^2+\epsilon^2}}\right) \right\rangle\\
&=2\re\left\langle \bm{p},F^*\left(\left(\bm{1}-\frac{\bm{M}}{2\sqrt{\lvert
       F(\bm{u})\rvert^2+\epsilon^2}}\circ\left(\frac{\lvert
       F(\bm{u})\rvert^2+2\epsilon^2}{\lvert
       F(\bm{u})\rvert^2+\epsilon^2}\right)\right)\circ
 F(\bm{q})\right)\right\rangle\nonumber\\
&\phantom{++++}+2\re\left\langle\bm{p},F^*\left(\frac{F(\bm{u})^2\circ \bm{M}}{2\left(\lvert
    F(\bm{u})\rvert^2+\epsilon^2\right)^{3/2}}\circ \overline{F(\bm{q})}\right)\right\rangle.
\end{align*}

So gradient and Hessian-vector multiplication of LS model are given by
\eqref{eq:32} and~\eqref{eq:33}.

\printbibliography
\end{document}

%%% Local Variables: 
%%% mode: latex
%%% TeX-master: t
%%% reftex-default-bibliography:("/home/leechi/Documents/wavefront/paper/phaseretrieval.bib")
%%% TeX-engine: xetex 
%%% End: 